\documentclass[12pt]{amsart}

\usepackage{amsmath}
\usepackage{amssymb}
\usepackage{longtable}
\usepackage{diagbox}

\theoremstyle{plain}
\newtheorem{lemma}{Lemma}[section]
\newtheorem{proposition}[lemma]{Proposition}
\newtheorem{theorem}[lemma]{Theorem}
\newtheorem{corollary}[lemma]{Corollary}
\newtheorem{problem}[lemma]{Problem}

\theoremstyle{definition}

\newtheorem{remark}[lemma]{Remark}
\newtheorem{example}[lemma]{Example}

\title{Generating function on epimorphisms between $2$-bridge knot groups} 

\author{Masaaki Suzuki}

\address{Department of Frontier Media Science, Meiji University}

\email{macky@fms.meiji.ac.jp}

\begin{document}

\begin{abstract}
We have the generating function which 
determines the number of $2$-bridge knot groups admitting epimorphisms 
onto the knot group of a given $2$-bridge knot, in terms of crossing number. 
In this paper, we will refine this formula by taking account into genus as well as crossing number.  
Next, we determine the number of epimorphisms between fibered $2$-bridge knot groups. 
Moreover, we discuss degree one maps and $2$-bridge knots uknotting number one. 
\end{abstract}

\maketitle
\section{Introduction}\label{sect:intro}

Let $K$ be a knot and $G(K)$ the knot group, namely, 
the fundamental group of the exterior of $K$ in $S^3$. 
Recently, many papers have studied epimorphisms between knot groups. 
One of the most important problems (called Simon's conjecture),   
every knot group maps onto at most finitely many knot groups, 
is proved affirmatively by  Agol and Liu in \cite{agolliu}. 

In this paper, we focus on $2$-bridge knots. 
Boileau, Boyer, Reid, and Wang in \cite{BBRW} showed that 
Simon's conjecture is true for $2$-bridge knots, 
previous to \cite{agolliu}. 
In \cite{BBRW}, they also showed that 
if a $2$-bridge knot group admits an epimorphism onto another knot group $G(K)$, 
then $K$ is a $2$-bridge knot or the trivial knot. 
Ohtsuki, Riley, and Sakuma in \cite{ORS} gave us  
a systematic construction of epimorphisms between $2$-bridge knot groups. 
Conversely, as a consequence of Agol's result announced in \cite{agol}, 
all epimorphisms between $2$-bridge knot groups arise from this construction. 
Besides, Aimi, Lee, Sakai, and Sakuma also gave an alternative proof in \cite{ALSS}.

By using these fundamental results, 
the previous paper \cite{SZK} showed the relationship between 
the existence of an epimorphism between $2$-bridge knot groups 
and their crossing numbers. 
This coincides with the results of \cite{kitano-suzuki1} and \cite{HKMS}, 
which enumerate all the epimorphisms with up to $11$ crossings. 
Moreover, the same paper \cite{SZK} formulated the generating function 
which determines the number of $2$-bridge knots $K$ 
admitting epimorphisms from the knot group $G(K)$ onto the knot group of a given $2$-bridge knot. 
In \cite{suzukitran}, 
we obtained an explicit condition on genera of $2$-bridge knots 
such that there exists an epimorphism between their knot groups. 
 
In this paper, we will refine the previous generating function in terms of genus as well as crossing number.  
Next, the number of epimorphisms between fibered $2$-bridge knot groups 
will be determined by an explicit formula. 
Furthermore, we will discuss epimorphisms induced by degree one maps 
between $2$-bridge knot groups 
and will give many examples of epimorphisms between $2$-bridge knot groups 
with unknotting number one. 

Throughout this paper, we do not distinguish a knot from its mirror image, 
since their knot groups are isomorphic and we discuss epimorphisms between knot groups. 
We do not consider an epimorphism from a knot group onto itself, 
namely, we deal with only non-isomorphic epimorphisms.  
The numberings of the knots follow Rolfsen's book \cite{rolfsen} and  
the data of knots can be obtained by KnotInfo \cite{CL}. 


\section{$2$-bridge knot and continued fraction expansion}\label{sect:cfe}

In this section, we recall some well-known results on $2$-bridge knots. 
See \cite{bzh}, \cite{murasugi} in detail, for example. 

A $2$-bridge knot corresponds to a rational number $r = q/p \in {\mathbb Q} \cap (0,1)$, where $p$ is odd. 
Then we denote by $K(q/p)$ such a $2$-bridge knot. 
A rational number $q/p$ can 
be expressed as a continued fraction expansion:
\[
 \frac{q}{p} = [a_1,a_2, \ldots, a_{m-1},a_m] = 
\frac{1}{a_1 + \frac{1}{a_2 + \frac{1}{\ddots \frac{1}{a_{m-1} + \frac{1}{a_m}}}}} .
\]

Schubert classified $2$-bridge knots as follows. 

\begin{theorem}[Schubert] 
Let $K(q/p)$ and $K(q'/p')$ be $2$-bridge knots. 
These knots are equivalent if and only if 
the following conditions hold. 
\begin{enumerate}
\item $p = p'$. 
\item Either $q \equiv \pm q' \pmod{p}$ or $q q' \equiv \pm 1 \pmod{p}$. 
\end{enumerate}
\end{theorem}

By using continued fraction expansion, we have 
\[
K([a_1,a_2,\ldots, a_m]) = K([\varepsilon a_m, \ldots ,\varepsilon a_2, \varepsilon a_1]) 
\]
where $\varepsilon = (-1)^{m-1}$. 
Moreover, the mirror image of $K([a_1,a_2,\ldots, a_m])$ is expressed as $K([- a_1,- a_2,\ldots, - a_m])$. 
In this paper, we do not distinguish between a knot and its mirror image, then 
\begin{align*}
&K([a_1,a_2,\ldots, a_m]), K([-a_1,-a_2,\ldots, -a_m]), K([a_m, \ldots ,a_2, a_1]), K([- a_m, \ldots ,- a_2, - a_1]) 
\end{align*}
can be considered as the same knot. 
In particular, we may assume that $a_1>0$ in $K([a_1,a_2,\ldots, a_m])$. 
A sequence $(a_1,a_2,\ldots,a_m)$ arisen from a continued fraction expansion is called {\it symmetric}, 
if $(a_1,a_2,\ldots,a_m) = (a_m,\ldots,a_2,a_1)$ for $a_m > 0$, 
or if $(a_1,a_2,\ldots,a_m) = (-a_m,\ldots,-a_2,-a_1)$ for $a_m < 0$.

We denote by $c(K)$ and $g(K)$ the crossing number of a knot $K$ and the genus of $K$ respectively. 
If a rational number $r$ is expressed as $[a_1,a_2,\ldots,a_m]$ where $a_i > 0$ and $a_1,a_m \geq 2$, 
then the continued fraction expansion is called {\it standard}. 
The crossing number of $K(r)$ for the standard continued fraction $r = [a_1,a_2,\ldots, a_m]$ 
is given by 
\[
 c(K([a_1,a_2,\ldots,a_m])) = \sum_{i=1}^m a_i . 
\]
Besides, if all the $a_i$'s are non-zero even numbers in $[a_1,a_2,\ldots,a_m]$, 
then the continued fraction expansion is called {\it even}. 
For a $2$-bridge knot $K(r)$, 
an even continued fraction expansion for $r$ is unique up to symmetry 
and the length, which is the number of the components, 
is even (see \cite{GHS}, \cite{HMS} detail).  
The genus of $K(r)$ for the even continued fraction $r = [a_1,a_2,\ldots, a_{2m}]$ 
is given by $m$.

\section{Epimorphisms between $2$-bridge knot groups}

We have the following remarkable result about epimorphisms between $2$-bridge knot groups. 
Namely, the rational numbers for these $2$-bridge knots have a certain relationship. 

\begin{theorem}[Ohtsuki-Riley-Sakuma \cite{ORS}, Agol \cite{agol}, Aimi-Lee-Sakai-Sakuma \cite{ALSS}]\label{thm:ors}
Let $K(r), K(\tilde{r})$ be $2$-bridge knots, 
where $r = [a_1,a_2,\ldots,a_m]$. 
There exists an epimorphism $\varphi : G(K(\tilde{r})) \to G(K(r))$, 
if and only if $\tilde{r}$ can be written as 
\[
\tilde{r} = 
[\varepsilon_1 {\bf a}, 2 c_1, 
\varepsilon_2 {\bf a}^{-1}, 2 c_2, 
\varepsilon_3 {\bf a}, 2 c_3, 
\varepsilon_4 {\bf a}^{-1}, 2 c_4, 
\ldots, 
\varepsilon_{2n} {\bf a}^{-1}, 2 c_{2n}, \varepsilon_{2n+1} {\bf a}] , 
\]
where ${\bf a} = (a_1, a_2,\ldots,a_m), {\bf a}^{-1} = (a_m, a_{m-1},\ldots,a_1)$, 
$\varepsilon_i = \pm 1 \, \, (\varepsilon_1 = 1)$, and $c_i \in {\mathbb Z}$. 
\end{theorem}

Moreover, it is also shown that 
epimorphisms between $2$-bridge knot groups are always meridional, 
though the previous paper \cite{chasuzuki} proves that 
some epimorphisms between knot groups are not always meridional. 
Here an epimorphism is called {\it meridional} if it sends meridian to meridian. 

If a rational number $\tilde{r}$ is expressed as 
\[
 \tilde{r} = 
[\varepsilon_1 {\bf a}, 2 c_1, 
\varepsilon_2 {\bf a}^{-1}, 2 c_2, 
\varepsilon_3 {\bf a}, 2 c_3, 
\varepsilon_4 {\bf a}^{-1}, 2 c_4, 
\ldots, 
\varepsilon_{2n} {\bf a}^{-1}, 2 c_{2n}, \varepsilon_{2n+1} {\bf a}] , 
\]
then we say that 
$\tilde{r}$ has a continued fraction expansion of {\it type} $2n+1$ with respect to ${\bf a} = (a_1, a_2,\ldots,a_m)$. 
We may assume that $c_i \neq 0$ or $\varepsilon_i \cdot \varepsilon_{i+1} \neq -1$, see \cite{SZK} in detail. 

We recall the following three results from the previous papers \cite{SZK} and \cite{suzukitran}. 
One of main purposes of this paper is to refine Theorem \ref{thm:generatingfunction} below 
by taking account into genus not only crossing number. 
Proposition \ref{thm:crossingnumber} and Proposition \ref{thm:suzukitran} are useful to prove this refinement. 

For a given $2$-bridge knot $K(r)$,  
the following generating function gives us the number of $2$-bridge knots $K(\tilde{r})$ 
which admit epimorphisms $\varphi : G(K(\tilde{r})) \to G(K(r))$, 
in terms of $c(K({\tilde r}))$.

\begin{theorem}[\cite{SZK}]\label{thm:generatingfunction}
Let $K(r)$ be a $2$-bridge knot with $c_r$ crossings. 
We take the standard continued fraction expansion  $[a_1,a_2,\ldots,a_m]$ of $r$ 
and define the generating function as follows: 
\[
 f_c(r) = \sum_{n=1}^\infty \sum_{k=0}^\infty \, \bar{f}_c(n,k) \, t^{(2n+1) c_r + k} , 
\]
\begin{itemize}
\item[(A)] if $(a_1,a_2,\ldots,a_m)$ is not symmetric, then 
\[
\bar{f}_c(n,k) = 2^{2n} \, 
\Big(
\begin{array}{c}
2n+ k - 1 \\ k
\end{array}
\Big) ,
\]
\item[(B)] if $(a_1,a_2,\ldots,a_m)$ is symmetric, then 
\[
 \bar{f}_c(n,k) = 
\left\{
\begin{array}{ll}
2^{2n-1} 
\Big(
\begin{array}{c}
2n+ k - 1 \\ k
\end{array}
\Big)
& k \mbox{ : odd} \\
2^{2n-1} 
\Big(
\begin{array}{c}
2n+ k - 1 \\ k
\end{array}
\Big)
+ 
2^{n-1} 
\Big(
\begin{array}{c}
n+ \frac{k}{2} - 1 \\ \frac{k}{2}
\end{array}
\Big)
& 
k \mbox{ : even} \\
\end{array}
\right. , 
\]
\end{itemize}
and $\displaystyle{
\Big(
\begin{array}{c}
a \\ b
\end{array}
\Big) 
= \frac{a!}{b! \, (a-b)!}
}$. 
Then the number of $2$-bridge knots $K$ with $c$ crossings  
which admit epimorphisms $\varphi : G(K) \to G(K(r))$ 
is the coefficient of $t^{c}$ in $f_c(r)$. 
\end{theorem}

In order to show Theorem \ref{thm:generatingfunction}, 
we determine the crossing number of $K(\tilde{r})$ 
by using a continued fraction expansion for $\tilde{r}$ of type $2n+1$. 

\begin{proposition}[\cite{SZK}]\label{thm:crossingnumber}
Let $[a_1,a_2,\ldots,a_m]$ be a standard continued fraction expansion. 
Suppose that a rational number $\tilde{r}$ has a continued fraction expansion 
of type $2n+1$ with respect to ${\bf a} = (a_1,a_2,\ldots,a_m)$: 
\[
 \tilde{r} = 
[\varepsilon_1 {\bf a}, 2 c_1, 
\varepsilon_2 {\bf a}^{-1}, 2 c_2, 
\varepsilon_3 {\bf a}, 2 c_3, 
\varepsilon_4 {\bf a}^{-1}, 2 c_4, 
\ldots, 
\varepsilon_{2n} {\bf a}^{-1}, 2 c_{2n}, \varepsilon_{2n+1} {\bf a}] . 
\]
We define $\psi (i), \bar{\psi} (i), \bar{c}_i$ for $\tilde{r}$ as follows: 
\[
\psi (i) = 
\left\{
\begin{array}{ll}
1 & \, \, \varepsilon_i \cdot c_i < 0 \\
0 & \, \, \varepsilon_i \cdot c_i \geq 0 
\end{array}
\right. , \quad
\bar{\psi} (i) = 
\left\{
\begin{array}{ll}
1 & \, \, c_i \cdot \varepsilon_{i+1} < 0 \\
0 & \, \, c_i \cdot \varepsilon_{i+1} \geq 0 
\end{array}
\right. , \quad 
\bar{c}_i = 2 |c_i| - \psi (i) - \bar{\psi} (i).
\]
Then $\bar{c}_i\geq 0$ and 
the crossing number of $K(\tilde{r})$ is given by 
\[
 c(K(\tilde{r})) = 
(2n + 1) \sum_{j=1}^m a_j + \sum_{i=1}^{2n} \bar{c}_i. 
\]
\end{proposition}

Remark that 
\[
 \sum_{i=1}^{2n} (\psi (i) + \bar{\psi} (i)) 
\]
indicates the number of sign changes in $\tilde{r}$. 
As a corollary, we get 
\[
c(K) \geq 3 c(K') , 
\]
if there exits an epimorphism $\varphi : G(K) \to G(K')$ for $2$-bridge knots $K,K'$. 

\begin{proposition}[\cite{suzukitran}]\label{thm:suzukitran}
Let $a_1,a_2, \ldots,a_{2m}$ be non-zero even numbers. 
Suppose that a rational number $\tilde{r}$ has a continued fraction expansion 
of type $2n+1$ with respect to ${\bf a} = (a_1,a_2,\ldots,a_{2 m})$: 
\[
 \tilde{r} = 
[\varepsilon_1 {\bf a}, 2 c_1, 
\varepsilon_2 {\bf a}^{-1}, 2 c_2, 
\varepsilon_3 {\bf a}, 2 c_3, 
\varepsilon_4 {\bf a}^{-1}, 2 c_4, 
\ldots, 
\varepsilon_{2n} {\bf a}^{-1}, 2 c_{2n}, \varepsilon_{2n+1} {\bf a}] . 
\]
Then the genus of $K(\tilde{r})$ is given by 
\[
g(K(\tilde{r})) = (2n+1) m + n - | \{ i \, |\,  c_i = 0 \} |  .
\]
In particular, the maximum genus and the minimum genus of $K(\tilde{r})$ are $(2n+1) m + n$ and $(2n+1) m - n$ 
respectively.  
\end{proposition}

As a corollary, we get 
\[
g(K) \geq 3 g(K') - 1 , 
\]
if there exits an epimorphism $\varphi : G(K) \to G(K')$ for $2$-bridge knots $K,K'$.

\section{Generating function with respect to crossing number and genus}

\begin{theorem}\label{thm:generatingfunction2}
For a $2$-bridge knot $K(r)$ of genus $g_r$ with $c_r$ crossings, 
we take the even continued fraction expansion  $[a_1,a_2,\ldots,a_{2 g_r}]$ of $r$ 
and define the generating function in two variables $s,t$ as follows: 
\[
 f(r) = \sum_{n=1}^\infty \sum_{l=-n}^n \sum_{k=0}^\infty \, \bar{f} (n,l,k) \, 
s^{(2n+1) g_r + l}
\, 
t^{(2n+1) c_r + k},
\]
\begin{itemize}
\item[(A)] if $(a_1,a_2,\ldots,a_{2 g_r})$ is not symmetric, then 
\[
\bar{f} (n,l,k) = 
\left\{
\begin{array}{l}
\Big(
\begin{array}{c}
2n \\ n+l
\end{array}
\Big)
\hfill k = 0 \\
0 \hfill k \neq 0, n+ l = 0 \\
\Big(
\begin{array}{c}
2n \\ n+l
\end{array}
\Big)
\left(
\displaystyle{\sum_{p=1}^{\min (k,n+l)}}
2^p \, 
\Big(
\begin{array}{c}
n + l \\ p
\end{array}
\Big)
\Big(
\begin{array}{c}
k - 1 \\ p - 1
\end{array}
\Big)
\right) \\
\hfill k \neq 0, n+ l \neq 0 
\end{array}
\right. , 
\]
\item[(B)] if $(a_1,a_2,\ldots,a_{2 g_r})$ is symmetric, then 
\[
\bar{f} (n,l,k) = 
\left\{
\begin{array}{l}
\frac{1}{2}
\Big(
\begin{array}{c}
2n \\ n+l
\end{array}
\Big)
\hfill k = 0, n + l \mbox{:odd} \\
\frac{1}{2} 
\Big(
\begin{array}{c}
2n \\ n+l
\end{array}
\Big) 
+ \frac{1}{2} 
\Big(
\begin{array}{c}
n \\ \frac{n+l}{2}
\end{array}
\Big)
\hfill k = 0, n + l \mbox{:even}  \\
0 \hfill k \neq 0, n + l = 0  \\
\Big(
\begin{array}{c}
2n \\ n+l
\end{array}
\Big)
\left(
\displaystyle{\sum_{p=1}^{\min (k,n+l)}}
2^{p-1} \, 
\Big(
\begin{array}{c}
n + l \\ p
\end{array}
\Big)
\Big(
\begin{array}{c}
k - 1 \\ p - 1
\end{array}
\Big)
\right) \\
\hfill k \neq 0,  ( n + l \mbox{:odd or } k \mbox{:odd}) \\
\Big(
\begin{array}{c}
2n \\ n+l
\end{array}
\Big)
\left(
\displaystyle{\sum_{p=1}^{\min (k,n+l)}}
2^{p-1} \, 
\Big(
\begin{array}{c}
n + l \\ p
\end{array}
\Big)
\Big(
\begin{array}{c}
k - 1 \\ p - 1
\end{array}
\Big)
\right) \\
+ \Big(
\begin{array}{c}
n \\ \frac{n+l}{2}
\end{array}
\Big)
\left(
\displaystyle{\sum_{p=1}^{\min (\frac{k}{2},\frac{n+l}{2})}}
2^{p-1} \, 
\Big(
\begin{array}{c}
\frac{n + l}{2} \\ p
\end{array}
\Big)
\Big(
\begin{array}{c}
\frac{k}{2} - 1 \\ p - 1
\end{array}
\Big)
\right) \qquad \qquad \\
\hfill k \neq 0,  n + l \neq 0, n + l \mbox{:even}, k \mbox{:even}
\end{array}
\right.  .
\]
\end{itemize}
Then the number of $2$-bridge knots $K$ 
of genus $g$ with $c$ crossings 
which admit epimorphisms $\varphi : G(K) \to G(K(r))$ 
is the coefficient of $s^{g} \, t^{c}$ in $f(r)$. 
\end{theorem}

\begin{proof}
We will count the number of $2$-bridge knots $K = K(\tilde{r})$ of genus $(2n+1) g_r + l$ with $(2n+1) c_r + k$ crossings
which correspond to a continued fraction expansion  
\begin{align} \label{rtilde}
 \tilde{r} = 
[\varepsilon_1 {\bf a}, 2 c_1, 
\varepsilon_2 {\bf a}^{-1}, 2 c_2, 
\varepsilon_3 {\bf a}, 2 c_3, 
\varepsilon_4 {\bf a}^{-1}, 2 c_4, 
\ldots, 
\varepsilon_{2n} {\bf a}^{-1}, 2 c_{2n}, \varepsilon_{2n+1} {\bf a}], 
\end{align}
where ${\bf a} = (a_1,a_2,\ldots,a_{2 g_r})$. 
We extend the definition of $\psi (i),\bar{\psi} (i)$ as follows: 
\[
\psi (i) = 
\left\{
\begin{array}{ll}
1 & \, \, \varepsilon_i \cdot a_{2 g_r} \cdot c_i < 0 \\
0 & \, \, \varepsilon_i \cdot a_{2 g_r} \cdot c_i \geq 0 
\end{array}
\right. , \quad
\bar{\psi} (i) = 
\left\{
\begin{array}{ll}
1 & \, \, c_i \cdot \varepsilon_{i+1} \cdot a_{2 g_r} < 0 \\
0 & \, \, c_i \cdot \varepsilon_{i+1} \cdot a_{2 g_r} \geq 0 
\end{array}
\right. .
\]
Since the last element of any standard continued fraction expansion is positive, 
this definition does not conflict with the definition of $\psi (i),\bar{\psi} (i)$ in Proposition \ref{thm:crossingnumber}. 
By Proposition \ref{thm:crossingnumber} and a similar argument, we have 
\[
\bar{c}_i = 2 |c_i| - \psi (i) - \bar{\psi} (i) \geq 0, \qquad 
 k = \sum_{i=1}^{2n} \bar{c}_i = \sum_{i=1}^{2n} 2 |c_i| - \psi (i) - \bar{\psi} (i).
\]
All the entries of (\ref{rtilde}) are even and the length of this continued fraction is $2(2n+1) g_r + 2n$. 
However, the genus is $(2n+1) g_r + l$. 
Then $n-l$ elements of $c_1, c_2, \ldots, c_{2n}$ are zero. 
It follows that 
$n-l$ elements of $\bar{c}_1, \bar{c}_2, \ldots, \bar{c}_{2n}$ are also zero
and that the remaining $\bar{c}_{j_1}, \bar{c}_{j_2},\ldots, \bar{c}_{j_{n+l}}$ are non-zero,  
and that we have $\displaystyle{\Big(
\begin{array}{c}
2n \\ n+l
\end{array}
\Big)}$ cases. 
Furthermore, 
we determine the number of ways to assign non-negative integers, whose total is $k$,  
to $(\bar{c}_{j_1}, \bar{c}_{j_2},\ldots, \bar{c}_{j_{n+l}})$. 
Besides, when we assign a non-zero number to $\bar{c}_j$, 
there are $2$ possibilities for each $c_j$, 
however, there is only $1$ possibility for $c_j$ if $\bar{c}_j = 0$
(see \cite[Proof for Theorem 5.1]{SZK}). 
Then the number of possibilities is 
\[
\left\{
\begin{array}{ll}
\Big(
\begin{array}{c}
2n \\ n+l
\end{array}
\Big)
& k = 0 \\
\Big(
\begin{array}{c}
2n \\ n+l
\end{array}
\Big)
\left(
\displaystyle{\sum_{p=1}^{\min (k,n+l)}}
2^p \, 
\Big(
\begin{array}{c}
n + l \\ p
\end{array}
\Big)
\Big(
\begin{array}{c}
k - 1 \\ p - 1
\end{array}
\Big)
\right)
& k \neq 0 
\end{array}
\right. ,
\]
since we choose $p$ components from $(\bar{c}_{j_1}, \bar{c}_{j_2},\ldots, \bar{c}_{j_{n+l}})$ and 
distribute non-zero numbers, whose total is $k$, to the $p$ components. 
Hence we get the desired generating function for the case (A). 

If a continued fraction expansion of $r$ is symmetric, namely, case (B), we have 
\[
K([\varepsilon_1 {\bf a}, 2 c_1, \varepsilon_2 {\bf a}^{-1}, 
\ldots, 
2 c_{2n}, \varepsilon_{2n+1} {\bf a}])
=
K([\varepsilon_{2n+1} {\bf a}, 2 c_{2n}, 
\ldots, 
\varepsilon_{2} {\bf a}^{-1}, 2 c_1, \varepsilon_1 {\bf a}]) .
\]
It implies that we counted the same knot twice for the case that 
a continued fraction expansion of $\tilde{r}$ are is not symmetric, that is, $\tilde{r}$ is not in the form 
\[
 [\varepsilon_1 {\bf a}, 2 c_1, \ldots, 
2 c_n, \varepsilon_{n+1} {\bf a}^{\pm 1}, 2 c_n, \ldots, 
2 c_1, \varepsilon_1 {\bf a}] . 
\]
When $k=0$ and $n+l$ is odd, all the continued fraction expansions are not in this form 
and then the number is half of the case (A). 
When $k=0$ and $n+l$ is even, the number of knots is 
\[
\frac{
\Big(
\begin{array}{c}
2n \\ n+l
\end{array}
\Big) - 
\Big(
\begin{array}{c}
n \\ \frac{n+l}{2} 
\end{array}
\Big)
}{2} + 
\Big(
\begin{array}{c}
n \\ \frac{n+l}{2}
\end{array}
\Big) 
= 
\frac{1}{2}
\Big(
\begin{array}{c}
2n \\ n+l
\end{array}
\Big)
+
\frac{1}{2}
\Big(
\begin{array}{c}
n \\ \frac{n+l}{2}
\end{array}
\Big).
\]
Similarly, we obtain the remaining of the generating function of (B) and it completes the proof. 
\end{proof}

\begin{example}\label{ex:3162}
First, we apply Theorem \ref{thm:generatingfunction2} to the trefoil knot $3_1 = K(1/3) = K([2,-2])$. 
The generating function for the trefoil knot is 
\begin{align*}
 & f(1/3) = \\
& \, \,  s^2 t^9 + s^3 t^9 + s^4 t^9 + 2 s^3 t^{10} + 2 s^4 t^{10} + 2 s^3 t^{11} + 5 s^4 t^{11} \\ 
& + 2 s^3 t^{12} + 6 s^4 t^{12} + 2 s^3 t^{13} + 9 s^4 t^{13} + 2 s^3 t^{14} + 10 s^4 t^{14} \\
& + 3 s^3 t^{15} + 15 s^4 t^{15} + 4 s^5 t^{15} + 2 s^6 t^{15} + s^7 t^{15} \\
& + 2 s^3 t^{16} + 18 s^4 t^{16} + 12 s^5 t^{16} + 12 s^6 t^{16} + 4 s^7 t^{16} \\
& + 2 s^3 t^{17} + 21 s^4 t^{17} + 26 s^5 t^{17} + 36 s^6 t^{17} + 18 s^7 t^{17} \\ 
& + 2 s^3 t^{18} + 22 s^4 t^{18} + 36 s^5 t^{18} + 76 s^6 t^{18} + 44 s^7 t^{18} \\ 
& + 2 s^3 t^{19} + 25 s^4 t^{19} + 50 s^5 t^{19} + 132 s^6 t^{19} + 100 s^7 t^{19} \\
& + 2 s^3 t^{20} + 26 s^4 t^{20} + 60 s^5 t^{20} + 204 s^6 t^{20} + 180 s^7 t^{20} \\
& + 2 s^3 t^{21} + 30 s^4 t^{21} + 77 s^5 t^{21} + 301 s^6 t^{21} + 320 s^7 t^{21} + 9 s^8 t^{21} + 3 s^9 t^{21} + s^{10} t^{21} \\
& + 2 s^3 t^{22} + 30 s^4 t^{22} + 90 s^5 t^{22} + 426 s^6 t^{22} + 536 s^7 t^{22} + 60 s^8 t^{22} + 30 s^9 t^{22} + 6 s^{10} t^{22} \\
& + \cdots 
\end{align*}
and provides Table \ref{table31}, which shows the coefficients of $s^g \, t^c$ in $f(1/3)$. 

\begin{table}[h]
\begin{tabular}{c||c|c|c|c|c|c|c|c|c|c|c|c|c|c|c}
\diagbox{$g$}{$c$} & 8 & 9 & 10 & 11 & 12 & 13 & 14 & 15 & 16 & 17 & 18 & 19 & 20 & 21 & 22\\
\hline
1 & 0 & 0 & 0 & 0 & 0 & 0 & 0 & 0 & 0 & 0 & 0 & 0 & 0 & 0 & 0 \\
\hline
2 & 0 & 1 & 0 & 0 & 0 & 0 & 0 & 0 & 0 & 0 & 0 & 0 & 0 & 0 & 0 \\
\hline
3 & 0 & 1 & 2 & 2 & 2 & 2 & 2 & 3 & 2 & 2 & 2 & 2 & 2 & 2 & 2 \\
\hline
4 & 0 & 1 & 2 & 5 & 6 & 9 & 10 & 15 & 18 & 21 & 22 & 25 & 26 & 30 & 30 \\
\hline
5 & 0 & 0 & 0 & 0 & 0 & 0 & 0 & 4 & 12 & 26 & 36 & 50 & 60 & 77 & 90 \\
\hline
6 & 0 & 0 & 0 & 0 & 0 & 0 & 0 & 2 & 12 & 36 & 76 & 132 & 204 & 301 & 426 \\
\hline
7 & 0 & 0 & 0 & 0 & 0 & 0 & 0 & 1 & 4 & 18 & 44 & 100 & 180 & 320 & 536 \\
\hline
8 & 0 & 0 & 0 & 0 & 0 & 0 & 0 & 0 & 0 & 0 & 0 & 0 & 0 & 9 & 60 \\
\hline
9 & 0 & 0 & 0 & 0 & 0 & 0 & 0 & 0 & 0 & 0 & 0 & 0 & 0 & 3 & 30 \\
\hline
10 & 0 & 0 & 0 & 0 & 0 & 0 & 0 & 0 & 0 & 0 & 0 & 0 & 0 & 1 & 6 \\
\hline
11 & 0 & 0 & 0 & 0 & 0 & 0 & 0 & 0 & 0 & 0 & 0 & 0 & 0 & 0 & 0 
\end{tabular}
\caption{onto $3_1$}
\label{table31}
\end{table}

The previous paper \cite{kitano-suzuki1} shows that 
the knot groups of $2$-bridge knot with $9$ crossings 
admitting epimorphisms onto the trefoil knot group are  those of $9_1,9_6,9_{23}$, 
whose genera are $4,3,2$ respectively. 
Similarly, the knot groups of $2$ (respectively $2$) distinct $2$-bridge knots with $10$ crossings 
of genus $3$ (respectively $4$) admitting epimorphisms onto the trefoil knot group 
are those of $10_{32},10_{40}$ (respectively $10_5,10_9$).  

Next, we consider $6_2 = K(7/11) =  K([2,2,-2,2])$ as another example. 
Since $(2,2,-2,2)$ is not symmetric, the generating function for $K(7/11)$ is  
\begin{align*}
& f(7/11) = \\
& \, s^5 t^{18} + 2 s^6 t^{18} + s^7 t^{18} + 4 s^6 t^{19} + 4 s^7 t^{19} +  4 s^6 t^{20} + 8 s^7 t^{20} + 4 s^6 t^{21} + 12 s^7 t^{21} \\
& \,  + 4 s^6 t^{22} + 16 s^7 t^{22} + 4 s^6 t^{23} + 20 s^7 t^{23} + 4 s^6 t^{24} + 24 s^7 t^{24} + 4 s^6 t^{25} + 28 s^7 t^{25} \\
& \,  + 4 s^6 t^{26} + 32 s^7 t^{26} + 4 s^6 t^{27} + 36 s^7 t^{27} + 4 s^6 t^{28} + 40 s^7 t^{28} + 4 s^6 t^{29} + 44 s^7 t^{29} \\
& \, + 4 s^6 t^{30} + 48 s^7 t^{30} + s^8 t^{30} + 4 s^9 t^{30} + 6 s^{10} t^{30} + 4 s^{11} t^{30} + s^{12} t^{30} \\
& \, +4 s^6 t^{31} + 52 s^7 t^{31} + 8 s^9 t^{31} + 24 s^10 t^{31} + 24 s^{11} t^{31} + 8 s^{12} t^{31} + \cdots .
\end{align*}

Table \ref{table62} shows the coefficients of $s^g t^c$ in $f(7/11)$. 

\begin{table}[h]
\begin{tabular}{c||c|c|c|c|c|c|c|c|c|c|c|c|c|c|c}
\diagbox{$g$}{$c$} & 17 & 18 & 19 & 20 & 21 & 22 & 23 & 24 & 25 & 26 & 27 & 28 & 29 & 30 & 31 \\
\hline
4 & 0 & 0 & 0 & 0 & 0 & 0 & 0 & 0 & 0 & 0 & 0 & 0 & 0 & 0 & 0 \\
\hline
5 & 0 & 1 & 0 & 0 & 0 & 0 & 0 & 0 & 0 & 0 & 0 & 0 & 0 & 0 & 0 \\
\hline
6 & 0 & 2 & 4 & 4 & 4 & 4 & 4 & 4 & 4 & 4 & 4 & 4 & 4 & 4 & 4 \\
\hline
7 & 0 & 1 & 4 & 8 & 12 & 16 & 20 & 24 & 28 & 32 & 36 & 40 & 44 & 48 & 52 \\
\hline
8 & 0 & 0 & 0 & 0 & 0 & 0 & 0 & 0 & 0 & 0 & 0 & 0 & 0 & 1 & 0 \\
\hline
9 & 0 & 0 & 0 & 0 & 0 & 0 & 0 & 0 & 0 & 0 & 0 & 0 & 0 & 4 & 8 \\
\hline
10 & 0 & 0 & 0 & 0 & 0 & 0 & 0 & 0 & 0 & 0 & 0 & 0 & 0 & 6 & 24 \\
\hline
11 & 0 & 0 & 0 & 0 & 0 & 0 & 0 & 0 & 0 & 0 & 0 & 0 & 0 & 4 & 24 \\
\hline
12 & 0 & 0 & 0 & 0 & 0 & 0 & 0 & 0 & 0 & 0 & 0 & 0 & 0 & 1 & 8 \\
\hline
13 & 0 & 0 & 0 & 0 & 0 & 0 & 0 & 0 & 0 & 0 & 0 & 0 & 0 & 0 & 0 
\end{tabular}
\caption{onto $6_2$}
\label{table62}
\end{table}

\end{example}

\begin{remark}
The previous generating function $f_c(r)$ of Theorem \ref{thm:generatingfunction} is  
obtained by substituting $1$ for $s$ in $f(r)$ of Theorem \ref{thm:generatingfunction2}. 
For example, we have $f_c(1/3)$ and $f_c(7/11)$ as follows: 
\begin{align*}
f_c(1/3) =& \, 3 t^9 + 4 t^{10} + 7 t^{11} + 8 t^{12} + 11 t^{13} + 12 t^{14} + 25 t^{15} 
+ 48 t^{16} + 103 t^{17} + 180 t^{18} \\ 
& + 309 t^{19} + 472 t^{20} + 743 t^{21} + 1180 t^{22} + 
\cdots ,
\end{align*}
\begin{align*}
f_c(7/11) =& \, 
4 t^{18} + 8 t^{19} + 12 t^{20} + 16 t^{21} + 20 t^{22} + 24 t^{23} + 28 t^{24} + 
 32 t^{25} + 36 t^{26} + 40 t^{27} \\
& + 44 t^{28} + 48 t^{29} + 68 t^{30} + 120 t^{31} +  \cdots . 
\end{align*}
Each coefficient is consistent with the total of each column in Table \ref{table31} and Table \ref{table62}. 
\end{remark}

We denote by $h(r,g,c)$ the coefficient of $s^g \, t^c$ in $f(r)$. 
In other words, the generating function $f(r)$ can be written as a formal power series in two variables $s,t$ :
\[
f(r) = \sum_{g=0}^\infty \sum_{c=0}^\infty \, h (r,g,c) \, s^{g} \, t^{c}  .
\]
For example, Table \ref{table62} shows that 
\begin{align*}
& h(7/11,5,18) = 1, \quad h(7/11,6,18) = 2, \quad h(7/11,7,18) = 1, \\
& h(7/11,g,18) = 0 \quad \mbox{ if } g \leq 4 \mbox{ or } 8 \leq g . 
\end{align*}
We can determine $h(r,g,c)$ more explicitly in specific cases as follows. 

\begin{corollary}\label{cor:specificcase}
Let $g_r$ and $c_r$ be the genus and the crossing number 
of a $2$-bridge knot $K(r)$ respectively. 
We denote by $[a_1,a_2,\ldots,a_{2g_r}]$ the even continued fraction expansion of $r$. 
\begin{enumerate}
\item If $g < 3 g_r - 1$ or $c < 3 c_r$ or 
there exists a natural number $n$ such that $(2n+1) g_r + n < g$ and $c < (2n+3) c_r$, 
then $h(r,g,c) = 0$.
\item Suppose that $n \leq g_r$.  
\begin{itemize}
\item[(A)] If $(a_1,a_2, \ldots, a_{2g_r})$ is not symmetric, then  
\begin{align*}
& h(r,(2n+1) g_r - n,c) = 
 \left\{
\begin{array}{ll}
1
& c = (2n+1) c_r) \\
0 & \mbox{otherwise} 
\end{array}
\right. , \\
& h(r,(2n+1) g_r - n + 1,c) = 
 \left\{
\begin{array}{ll}
0 & c < (2n+1) c_r  \\
2n & c = (2n+1) c_r \\
4n & c > (2n+1) c_r   
\end{array}
\right. .
\end{align*}
\item[(B)] If $(a_1,a_2, \ldots, a_{2g_r})$ is symmetric, then  
\begin{align*}
& h(r,(2n+1) g_r - n,c) = 
 \left\{
\begin{array}{ll}
1
& c = (2n+1) c_r \\
0 & \mbox{otherwise} 
\end{array}
\right. , \\
& h(r,(2n+1) g_r - n + 1,c) = 
 \left\{
\begin{array}{ll}
0 & c < (2n+1) c_r \\
n & c = (2n+1) c_r \\
2 n & c > (2n+1) c_r 
\end{array}
\right. .
\end{align*}
\end{itemize}
\item Suppose that $\max \{(2n-1) g_r + n, (2n+1) g_r - n \} \leq g \leq (2n+1) g_r + n$.  
\begin{itemize}
\item[(A)] If $(a_1,a_2, \ldots, a_{2g_r})$ is not symmetric, then  
\begin{align*}
& h(r,g,(2n+1) c_r)  = 
\Big(
\begin{array}{c}
2n \\ n+g - (2n+1) g_r
\end{array}
\Big) , \\
& h(r,g,(2n+1) c_r+1) = 
2 (n+g - (2n+1) g_r) \cdot \Big(
\begin{array}{c}
2n \\ n+g - (2n+1) g_r
\end{array}
\Big) .
\end{align*}
\item[(B)] If $(a_1,a_2, \ldots, a_{2g_r})$ is symmetric, then  
\begin{align*}
& h(r,g,(2n+1) c_r) = \left\{
\begin{array}{l}
\frac{1}{2} \Big(
\begin{array}{c}
2n \\ n+g - (2n+1) g_r
\end{array}
\Big) \\
\qquad \qquad \qquad  \qquad \qquad n+g - (2n+1) g_r: \mbox{odd} \\
\frac{1}{2} \Big(
\begin{array}{c}
2n \\ n+g - (2n+1) g_r
\end{array}
\Big) 
+ \frac{1}{2} \Big(
\begin{array}{c}
n \\ \frac{n+g - (2n+1) g_r}{2}
\end{array}
\Big) \\
\qquad \qquad \qquad \qquad \qquad  n+g - (2n+1) g_r: \mbox{even} 
\end{array}
\right. , \\
& h(r,g,(2n+1) c_r+1) = 
(n+g - (2n+1) g_r) \cdot \Big(
\begin{array}{c}
2n \\ n+g - (2n+1) g_r
\end{array}
\Big) .
\end{align*}
\end{itemize}
\end{enumerate}
\end{corollary}

\begin{proof}
\begin{enumerate}
\item 
It is obtained by Proposition \ref{thm:crossingnumber} and Proposition \ref{thm:suzukitran}. 
\item The inequality $n \leq g_r$ induces  
\[
(2n - 1) g_r + (n - 1) < (2n + 1) g_r - n  .  
\]
This implies that the maximum genus of type $2n-1$ is less than the minimum genus of type $2n+1$ 
by Proposition \ref{thm:suzukitran}. 
Then we can compute $h(r,g,c)$ taking account into only type $2n+1$.
If $(a_1,a_2, \ldots, a_{2g_r})$ is not symmetric, then  
\begin{align*}
&h(r,(2n+1) g_r - n, (2 n+1) c_r) = \bar{f}(n,-n,0) 
= \Big(
\begin{array}{c}
2n \\ 0
\end{array}
\Big)
= 1 \\
&h(r,(2n+1) g_r - n + 1, (2 n+1) c_r) = \bar{f}(n,-n + 1,0) 
= \Big(
\begin{array}{c}
2n \\ 1
\end{array}
\Big)
= 2n  \\
&h(r,(2n+1) g_r - n + 1, (2 n+1) c_r + k) = \bar{f}(n,-n + 1,k) 
= \Big(
\begin{array}{c}
2n \\ 1
\end{array}
\Big) \cdot 2
= 4n  ,
\end{align*}
where $k > 0$. Similarly if $(a_1,a_2, \ldots, a_{2g_r})$ is symmetric, then  
\begin{align*}
&h(r,(2n+1) g_r - n, (2 n+1) c_r) = \bar{f}(n,-n,0) 
= \frac{1}{2} \Big(
\begin{array}{c}
2n \\ 0
\end{array}
\Big)
+ \frac{1}{2} \Big(
\begin{array}{c}
n \\ 0
\end{array}
\Big)
= 1 \\
&h(r,(2n+1) g_r - n + 1, (2 n+1) c_r) = \bar{f}(n,-n + 1,0) 
= \frac{1}{2} \Big(
\begin{array}{c}
2n \\ 1
\end{array}
\Big)
= n  \\
&h(r,(2n+1) g_r - n + 1, (2 n+1) c_r + k) = \bar{f}(n,-n + 1,k) 
= \Big(
\begin{array}{c}
2n \\ 1
\end{array}
\Big)
= 2n  ,
\end{align*}
where $k > 0$.
\item 
In this case, similarly we can compute $h(r,g,c)$ taking account into only type $2n+1$. Then 
\begin{align*}
&h(r,g,(2n+1) c_r) = 
\bar{f}(n, g - (2 n+1) g_r, 0), \\
&h(r,g,(2n+1) c_r+1)  = 
\bar{f}(n,g - (2n+1) g_r,1)  
\end{align*}
and the values of $\bar{f}(n,l,k)$ of Theorem \ref{thm:generatingfunction2} deduce the statement. 
\end{enumerate}
\end{proof}

If a $2$-bridge knot $K'$ and a natural number $c \in {\mathbb N}$ are given,  
then the number of $2$-bridge knots $K$ with $c$ crossings 
such that $G(K)$ admits an epimorphism onto $G(K')$ is finite. 
However, the corresponding statement for a given genus is quite different in general. 
More precisely, we obtain the following corollary 
by Proposition \ref{thm:suzukitran} and Theorem \ref{thm:generatingfunction2}. 

\begin{corollary}
For a given $2$-bridge knot $K'$ and a natural number $g \in {\mathbb N}$, 
the number of $2$-bridge knots $K$ of genus $g$ such that $G(K)$ admits an epimorphism onto $G(K')$ 
is 
\[
\left\{
\begin{array}{ll}
0 & g < 3 g(K') - 1 \mbox{ or } (2n+1) g(K') + n < g < (2n+3) g(K') - (n+1) \\
1 & (2n-1) g(K') - (n-1) < g = (2n+1) g(K') - n \\
\infty & \mbox{otherwise} 
\end{array}
\right. 
\]
for some $n \in {\mathbb N}$. 
\end{corollary}

\section{Fibered $2$-bridge knots}
A knot $K$ is called {\it fibered}, 
if the exterior of $K$ is a fiber bundle over the circle $S^1$. 
It is known that a $2$-bridge knot $K(r)$ is fibered 
if and only if all $a_i$'s are $\pm 2$ in the even continued fraction expansion $r = [a_1, a_2, \ldots, a_{2m}]$ 
(see \cite{crowell} \cite{murasugi} and also \cite{godasuzuki}). 
The number $h_f (r,g,c)$ of fibered $2$-bridge knots $K$ of genus $g$ with $c$ crossings 
which admit epimorphisms $\varphi : G(K) \to G(K(r))$ 
is the following.   

\begin{theorem}\label{thm:generatingfunctionfibered}
Let $c_r,g_r$ be the crossing number and the genus of a fibered $2$-bridge knot $K(r)$ respectively. 
We take the even continued fraction expansion $[a_1,a_2,\ldots,a_{2g_r}]$ of $r$.  
\begin{enumerate}
\item 
If there doesn't exists a natural number $n$ such that $g = (2n+1) g_r + n$ and $(2n+1) c_r \leq c \leq (2n+1) c_r + 4n$, 
then $h_f (r,g,c) = 0$. 
\item By using a natural number $n$ satisfying 
$g = (2n+1) g_r + n$ and $(2n+1) c_r \leq c \leq (2n+1) c_r + 4n$, 
\begin{itemize}
\item[(A)] if $(a_1,a_2,\ldots,a_{2g_r})$ is not symmetric, then 
\[
h_f (r,g,c) = 
\Big(
\begin{array}{c}
4n \\ c - (2n+1) c_r
\end{array}
\Big)  ,
\]
\item[(B)] if $(a_1,a_2,\ldots,a_{2g_r})$ is symmetric, then 
\[
h_f (r,g,c) = 
\left\{
\begin{array}{ll}
\frac{1}{2} \Big(
\begin{array}{c}
4n \\ c - (2n+1) c_r
\end{array}
\Big) &  c - (2n+1) c_r \mbox{ : odd} \\
\frac{1}{2} \Big(
\begin{array}{c}
4n \\ c - (2n+1) c_r
\end{array}
\Big)  
+ \frac{1}{2} \Big(
\begin{array}{c}
2n \\ \frac{c - (2n+1) c_r}{2}
\end{array}
\Big) & c - (2n+1) c_r \mbox{ : even}  
\end{array}
\right. .
\]
\end{itemize}
\end{enumerate}
Moreover, the total number of fibered $2$-bridge knots $K$ of genus $g$ 
which admit epimorphisms $\varphi : G(K) \to G(K(r))$ is 
\[
\left\{
\begin{array}{ll}
16^n & (a_1,a_2,\ldots,a_{2g_r}) \mbox{ is not symmetric}, ~ g = (2 n+1) g_r + n\\
\frac{4^n ( 4^n + 1)}{2} & (a_1,a_2,\ldots,a_{2g_r}) \mbox{ is symmetric}, ~ g = (2 n+1) g_r + n\\
0 & otherwise
\end{array}
\right. .
\]
\end{theorem}

\begin{proof}
We consider the fibered $2$-bridge knot $K=K(\tilde{r})$ for a rational number
\[
 \tilde{r} = 
[\varepsilon_1 {\bf a}, 2 c_1, 
\varepsilon_2 {\bf a}^{-1}, 2 c_2, 
\varepsilon_3 {\bf a}, 2 c_3, 
\varepsilon_4 {\bf a}^{-1}, 2 c_4, 
\ldots, 
\varepsilon_{2n} {\bf a}^{-1}, 2 c_{2n}, \varepsilon_{2n+1} {\bf a}], 
\]
where ${\bf a} = (a_1,a_2,\ldots,a_{2 g_r})$. 
Since $K(\tilde{r})$ is fibered, all $c_i$'s are $\pm 1$ 
and then the genus of $K(\tilde{r})$ is $(2n+1) g_r + n$. 
The crossing number of $K(\tilde{r})$ is 
\[
(2n+1) c_r + \sum_{i=1}^{2n} \left( 2 |c_i| - \psi (i) - \bar{\psi} (i) \right)
= (2n+1) c_r + 4 n - \sum_{i=1}^{2n} \left( \psi (i) + \bar{\psi} (i) \right) .
\] 
First, we see the case $(a_1,a_2,\ldots,a_{2 g_r})$ is not symmetric and $a_{2 g_r} >0$. 
If the signs change $k$ times in a sequence 
$(\varepsilon_1, c_1, \varepsilon_2, c_2, \varepsilon_3, c_3, \varepsilon_4, \ldots, c_{2n}, \varepsilon_{2n+1})$,
then the crossing number of $K(\tilde{r})$ is $(2n+1) c_r + 4 n - k$. 
There are  
\[
\left( 
\begin{array}{c}
4n \\
k
\end{array}
\right)
= 
\left( 
\begin{array}{c}
4n \\
4n - k
\end{array}
\right)
= 
\left( 
\begin{array}{c}
4n \\
c(K(\tilde{r})) - (2n+1) c_r
\end{array}
\right) 
\]
cases for such sequences. 
On the other hand, if $a_{2 g_r} <0$, 
the crossing number of $K(\tilde{r})$ is $(2n+1) c_r + 4 n - k$ for the case 
the signs change $4n - k$ times in the sequence. We have the same number of cases. 
Similarly, we can get $h_f (r,g,c)$ for symmetric case. 

Moreover, we see 
\begin{align*}
& \sum_{k=0}^{4n} 
\Big(
\begin{array}{c}
4n \\ k
\end{array}
\Big) = 2^{4n} = 16^n, \\ 
&
\sum_{k=0, k: odd}^{4n} 
\frac{1}{2} 
\Big(
\begin{array}{c}
4n \\ k
\end{array}
\Big)
+ \sum_{k=0, k:even}^{4n} 
\frac{1}{2} 
\left( \Big(
\begin{array}{c}
4n \\ k
\end{array}
\Big)
+ \frac{1}{2} 
\Big(
\begin{array}{c}
2n \\ \frac{k}{2}
\end{array}
\Big) \right) \\
& \qquad = \frac{1}{2} \left( 2^{4n-1} + ( 2^{4n-1} + 2^{2n} ) \right)
= \frac{4^n ( 4^n + 1)}{2} .
\end{align*}
Therefore we get the last part of the statement. 
\end{proof}

\begin{remark}
It is known that fibered knot groups admit epimorphisms only onto fibered knot groups, 
see \cite{silverwhitten} and \cite{kitano-suzuki3}. 
For any fixed natural number $g$, there are infinitely many $2$-bridge knots of genus $g$, 
however, there are finitely many fibered $2$-bridge knots of genus $g$.  
For example, fibered $2$-bridge knots of genus $1$ are only $3_1$ and $4_1$.  
\end{remark}

\begin{example}\label{ex:fibered3162}
Table \ref{tablefibered31} and \ref{tablefibered62} show 
that the number of fibered $2$-bridge knots which admit epimorphisms onto $3_1$ and $6_2$ respectively, 
according to Theorem \ref{thm:generatingfunctionfibered}.  

\begin{table}[h]
\begin{tabular}{c||c|c|c|c|c|c|c|c|c|c|c|c|c|c|c}
\diagbox{$g$}{$c$} & 8 & 9 & 10 & 11 & 12 & 13 & 14 & 15 & 16 & 17 & 18 & 19 & 20 & 21 & 22\\
\hline
1 & 0 & 0 & 0 & 0 & 0 & 0 & 0 & 0 & 0 & 0 & 0 & 0 & 0 & 0 & 0 \\
\hline
2 & 0 & 0 & 0 & 0 & 0 & 0 & 0 & 0 & 0 & 0 & 0 & 0 & 0 & 0 & 0 \\
\hline
3 & 0 & 0 & 0 & 0 & 0 & 0 & 0 & 0 & 0 & 0 & 0 & 0 & 0 & 0 & 0 \\
\hline
4 & 0 & 1 & 2 & 4 & 2 & 1 & 0 & 0 & 0 & 0 & 0 & 0 & 0 & 0 & 0 \\
\hline
5 & 0 & 0 & 0 & 0 & 0 & 0 & 0 & 0 & 0 & 0 & 0 & 0 & 0 & 0 & 0 \\
\hline
6 & 0 & 0 & 0 & 0 & 0 & 0 & 0 & 0 & 0 & 0 & 0 & 0 & 0 & 0 & 0 \\
\hline
7 & 0 & 0 & 0 & 0 & 0 & 0 & 0 & 1 & 4 & 16 & 28 & 38 & 28 & 16 & 4 \\
\hline
8 & 0 & 0 & 0 & 0 & 0 & 0 & 0 & 0 & 0 & 0 & 0 & 0 & 0 & 0 & 0 \\
\hline
9 & 0 & 0 & 0 & 0 & 0 & 0 & 0 & 0 & 0 & 0 & 0 & 0 & 0 & 0 & 0 \\
\hline
10 & 0 & 0 & 0 & 0 & 0 & 0 & 0 & 0 & 0 & 0 & 0 & 0 & 0 & 1 & 6 \\
\hline
11 & 0 & 0 & 0 & 0 & 0 & 0 & 0 & 0 & 0 & 0 & 0 & 0 & 0 & 0 & 0
\end{tabular}
\caption{Fibered onto $3_1$}
\label{tablefibered31}
\end{table}

\begin{table}[h]
\begin{tabular}{c||c|c|c|c|c|c|c|c|c|c|c|c|c|c|c}
\diagbox{$g$}{$c$} & 17 & 18 & 19 & 20 & 21 & 22 & 23 & 24 & 25 & 26 & 27 & 28 & 29 & 30 & 31 \\
\hline
4 & 0 & 0 & 0 & 0 & 0 & 0 & 0 & 0 & 0 & 0 & 0 & 0 & 0 & 0 & 0 \\
\hline
5 & 0 & 0 & 0 & 0 & 0 & 0 & 0 & 0 & 0 & 0 & 0 & 0 & 0 & 0 & 0 \\
\hline
6 & 0 & 0 & 0 & 0 & 0 & 0 & 0 & 0 & 0 & 0 & 0 & 0 & 0 & 0 & 0 \\
\hline
7 & 0 & 1 & 4 & 6 & 4 & 1 & 0 & 0 & 0 & 0 & 0 & 0 & 0 & 0 & 0 \\
\hline
8 & 0 & 0 & 0 & 0 & 0 & 0 & 0 & 0 & 0 & 0 & 0 & 0 & 0 & 0 & 0 \\
\hline
9 & 0 & 0 & 0 & 0 & 0 & 0 & 0 & 0 & 0 & 0 & 0 & 0 & 0 & 0 & 0 \\
\hline
10 & 0 & 0 & 0 & 0 & 0 & 0 & 0 & 0 & 0 & 0 & 0 & 0 & 0 & 0 & 0 \\
\hline
11 & 0 & 0 & 0 & 0 & 0 & 0 & 0 & 0 & 0 & 0 & 0 & 0 & 0 & 0 & 0 \\
\hline
12 & 0 & 0 & 0 & 0 & 0 & 0 & 0 & 0 & 0 & 0 & 0 & 0 & 0 & 1 & 8 \\
\hline
13 & 0 & 0 & 0 & 0 & 0 & 0 & 0 & 0 & 0 & 0 & 0 & 0 & 0 & 0 & 0 
\end{tabular}
\caption{Fibered onto $6_2$}
\label{tablefibered62}
\end{table}

\end{example}

In some case, all $2$-bridge knots whose knot groups admit an epimorphism onto another knot group are fibered. 

\begin{corollary}
Let $c_r,g_r$ be the crossing number and the genus of a fibered $2$-bridge knot $K(r)$ respectively. 
If there exists a natural number $n$ such that 
$g = (2n+1) g_r + n$ and $c = (2n+1) c_r$ or $(2n+1) c_r + 1$, 
then $h(r,g,c) = h_f(r,g,c)$, 
namely, all $2$-bridge knots whose knot groups admit an epimorphism onto $G(K(r))$ are fibered. 
\end{corollary}

\begin{proof}
First, we discuss the case $(a_1,a_2,\ldots,a_{2g_r})$ is not symmetric for $r = [a_1, a_2, \ldots, a_{2 g_r}]$. 
By Corollary \ref{cor:specificcase}, we have 
\begin{align*}
&h(r,(2n+1) g_r + n, (2n+1) c_r) = 
\left(
\begin{array}{c}
2n \\ n+ (2n+1) g_r + n (2n+1) g_r 
\end{array}
\right) 
= 
\left(
\begin{array}{c}
2n \\ 2 n 
\end{array}
\right) 
= 1,  \\
&h(r,(2n+1) g_r + n, (2n+1) c_r+1) = 
2 \cdot 2n \cdot 
\left(
\begin{array}{c}
2n \\ 2n 
\end{array}
\right) 
= 4n .
\end{align*}
On the other hand, Theorem \ref{thm:generatingfunctionfibered} induces 
\begin{align*}
&h_f(r,(2n+1) g_r + n, (2n+1) c_r) = 
\left(
\begin{array}{c}
4n \\ (2n+1) c_r - (2n+1) c_r 
\end{array}
\right) 
= 
\left(
\begin{array}{c}
4n \\ 0 
\end{array}
\right) 
= 1,  \\
&h_f(r,(2n+1) g_r + n, (2n+1) c_r+1) = 
\left(
\begin{array}{c}
4n \\ 1 
\end{array}
\right) 
= 4n .
\end{align*}
Similarly, if $(a_1,a_2,\ldots,a_{2g_r})$ is symmetric for $r = [a_1, a_2, \ldots, a_{2 g_r}]$, 
we see 
\begin{align*}
&h(r,(2n+1) g_r + n, (2n+1) c_r) = 
\frac{1}{2}\left(
\begin{array}{c}
2n \\ 2n 
\end{array}
\right) 
+ \frac{1}{2}
\left(
\begin{array}{c}
n \\ n 
\end{array}
\right) 
= 1,  \\
&h(r,(2n+1) g_r + n, (2n+1) c_r+1) = 
2n \cdot 
\left(
\begin{array}{c}
2n \\ 2n 
\end{array}
\right) 
= 2n ,\\
&h_f(r,(2n+1) g_r + n, (2n+1) c_r) = 
\frac{1}{2}\left(
\begin{array}{c}
4n \\ 0 
\end{array}
\right) 
+
\frac{1}{2}\left(
\begin{array}{c}
2n \\ 0 
\end{array}
\right) 
= 1,  \\
&h_f(r,(2n+1) g_r + n, (2n+1) c_r+1) = 
\frac{1}{2}\left(
\begin{array}{c}
4n \\ 1 
\end{array}
\right) 
= 2n .
\end{align*}
This completes the proof. 
\end{proof}

\section{Degree one map}\label{sect:degone}

In this section, we consider degree one maps between the exteriors of knots. 
Let $E(K)$ be the exterior of a knot $K$. 
If a map $f_* : H_3 (E(K), \partial E(K)) \to H_3(E(K'), \partial E(K'))$ induced by 
$f :  (E(K), \partial E(K)) \to (E(K'), \partial E(K'))$ sends the fundamental class 
$[E(K), \partial E(K)]$ to $\pm [E(K'), \partial E(K')]$, 
the map $f$ is called a {\it degree one map}. 
It is known that degree one map induces an epimorphism between their knot groups 
(see \cite{hempel}, also \cite{kitano-suzuki2}, \cite{kitano-suzuki4}). 
Boileau, Boyer, Reid, and Wang showed in \cite{BBRW} that 
any epimorphism from a hyperbolic $2$-bridge knot group onto a non-trivial knot group 
is induced by a non-zero degree map. 
Moreover, Gonzal\'{e}z-Ac\~{u}na, Ram\'{i}nez in \cite{GR} studied a similar property for 
epimorphisms from $2$-bridge knot groups to $(2,p)$ torus knot group. 

Combined with \cite[Remark 6.3 (1)]{ORS} and \cite[Theorem 8.3]{ALSS}, we have that 
if 
\[
\sum_{i=1}^{2n+1} \varepsilon_i = \pm 1 ,
\]
in Theorem \ref{thm:ors}, 
then the epimorphism is induced by degree one map. 
We denote by $h_1(r,g,c)$ the number of $2$-bridge knots 
whose knot groups admit epimorphisms onto $G(K(r))$ induced by degree one maps. 

\begin{proposition} 
Let $c_r,g_r$ be the crossing number and the genus of a $2$-bridge knot $K(r)$ respectively. 
We take the even continued fraction expansion $[a_1,a_2,\ldots,a_{2g_r}]$ of $r$.  
\begin{enumerate}
\item For a natural number $n$ satisfying $g > (2n-1) g_r + (n-1)$, we have
\begin{align*}
&h_1(r,g,(2n+1) c_r) = 0, \\
&h_1(r,g,(2n+1) c_r + 1) = \frac{1}{n} h(r,g,(2n+1) c_r + 1).
\end{align*}
\item For a natural number $n$ satisfying $n \leq g_r$, we have  
\begin{align*}
&h_1(r,(2n+1) g_r -n,c) = 0, \\
&h_1(r,(2n+1) g_r -n+1),c)  \\
& =  
\left\{
\begin{array}{ll}
0 & c - (2n+1) c_r \mbox{ is negative or even} \\
4 &  (a_1,a_2,\ldots,a_{2g_r}) \mbox{ is not symmetric}, ~ c - (2n+1) c_r \mbox{ is positive and odd} \\
2 &  (a_1,a_2,\ldots,a_{2g_r}) \mbox{ is symmetric}, ~ c - (2n+1) c_r \mbox{ is positive and odd} 
\end{array}
\right. .
\end{align*}
\end{enumerate}
\end{proposition}

\begin{proof}
Let $K(\tilde{r})$ be a $2$-bridge knot corresponding to 
a continued fraction expansion of type $2n+1$ : 
\begin{align*}
 \tilde{r} = 
[\varepsilon_1 {\bf a}, 2 c_1, 
\varepsilon_2 {\bf a}^{-1}, 2 c_2, 
\varepsilon_3 {\bf a}, 2 c_3, 
\varepsilon_4 {\bf a}^{-1}, 2 c_4, 
\ldots, 
\varepsilon_{2n} {\bf a}^{-1}, 2 c_{2n}, \varepsilon_{2n+1} {\bf a}], 
\end{align*}
with respect to ${\bf a} = (a_1,a_2,\ldots,a_{2 g_r})$. 
Suppose that an epimorphism from $G(K(\tilde{r}))$ onto $G(K(r))$ is induced by a degree one map. 
\begin{enumerate}
\item 
By $g > (2n-1) g_r + (n-1)$, we consider only a continued fraction expansion of type $2n+1$. 
If the crossing number of $c(K(\tilde{r}))$ is $(2n+1) c_r$,  then 
\[
 k = \sum_{i=1}^{2n} \bar{c}_i = \sum_{i=1}^{2n} 2 |c_i| - \psi (i) - \bar{\psi} (i) = 0.
\]
It turns out that all the $\bar{c}_i$'s are zero. 
By the argument of \cite[Proof of Theorem 5.1]{SZK}, all the $\varepsilon_i$'s are $+ 1$. 
Therefore it follows that the degree is $2n+1$. 

If the crossing number of $K(\tilde{r})$ is $(2n+1) c_r + 1$, then $k=1$. 
It turns out that only one element of $(\bar{c}_1, \bar{c}_2,\ldots,\bar{c}_{2n})$ is $1$ and the others are zero. 
Similarly, for non-zero $\bar{c}_i$, the degree is $2i - 2n - 1$. 
Then the degree is $1$ for the case either $\bar{c}_n$ or $\bar{c}_{n+1}$ is $1$. 
Therefore we get
\[
h_1(r,g,(2n+1) c_r + 1) = \frac{2}{2n} h(r,g,(2n+1) c_r + 1) .
\]
\item By the assumption $g_r \geq n$, we see that 
\[
(2n+1) g_r - n+1 < (2n+3) g_r - n. 
\]
Then it is enough to consider only a continued fraction expansion of type $2n+1$ again. 
First, we get $h_1(r,(2n+1) g_r -n,c) = 0$ 
by $(1)$ and Corollary \ref{cor:specificcase} (2). 

If the genus of $K(\tilde{r})$ is $(2n+1) g_r - n + 1$, then 
only one element of $(\bar{c}_1, \bar{c}_2,\ldots,\bar{c}_{2n})$ is non-zero and the others are zero. 
Suppose that $\bar{c}_i$ is non-zero.
If $\bar{c}_i$ is even, then the degree is $2n+1$ again. 
If $\bar{c}_i$ is odd, by the same argument, the degree is $2i - 2n - 1$. 
Moreover, the crossing number of $K(\tilde{r})$ is $(2n+1) + \bar{c}_i$. 
Therefore it deduces 
\[
h_1(r,(2n+1) g_r -n+ 1,c) = \frac{2}{2n} h(r,(2n+1) g_r -n+ 1,c) .
\]
We have the explicit number of the right hand side by Corollary \ref{cor:specificcase}. 
This completes the proof. 
\end{enumerate}
\end{proof}

\begin{corollary}
If there exits a degree one map $f :  (E(K), \partial E(K)) \to (E(K'), \partial E(K'))$
for $2$-bridge knots $K,K'$, 
then  
\[
c(K) \geq 3 c(K') + 1 , \qquad  g(K) \geq 3 g(K') . 
\]
\end{corollary}

As concluding examples, we obtain the explicit numbers of $h_1(1/3,g,c)$ and $h(7/11,g,c)$, 
which are shown in Table \ref{tabledegreeone31} and Table \ref{tabledegreeone62} . 

\begin{table}[h]
\begin{tabular}{c||c|c|c|c|c|c|c|c|c|c|c|c|c|c|c}
\diagbox{$g$}{$c$} & 8 & 9 & 10 & 11 & 12 & 13 & 14 & 15 & 16 & 17 & 18 & 19 & 20 & 21 & 22\\
\hline
1 & 0 & 0 & 0 & 0 & 0 & 0 & 0 & 0 & 0 & 0 & 0 & 0 & 0 & 0 & 0 \\
\hline
2 & 0 & 0 & 0 & 0 & 0 & 0 & 0 & 0 & 0 & 0 & 0 & 0 & 0 & 0 & 0 \\
\hline
3 & 0 & 0 & 2 & 0 & 2 & 0 & 2 & 0 & 2 & 0 & 2 & 0 & 2 & 0 & 2 \\
\hline
4 & 0 & 0 & 2 & 3 & 6 & 4 & 10 & 7 & 16 & 8 & 20 & 11 & 24 & 12 & 28 \\
\hline
5 & 0 & 0 & 0 & 0 & 0 & 0 & 0 & 0 & 6 & 7 & 18 & 12 & 30 & 19 & 44 \\
\hline
6 & 0 & 0 & 0 & 0 & 0 & 0 & 0 & 0 & 6 & 12 & 46 & 48 & 126 & 108 & 256 \\
\hline
7 & 0 & 0 & 0 & 0 & 0 & 0 & 0 & 0 & 2 & 7 & 30 & 46 & 130 & 149 & 370 \\
\hline
8 & 0 & 0 & 0 & 0 & 0 & 0 & 0 & 0 & 0 & 0 & 0 & 0 & 0 & 0 & 20 \\
\hline
9 & 0 & 0 & 0 & 0 & 0 & 0 & 0 & 0 & 0 & 0 & 0 & 0 & 0 & 0 & 10 \\
\hline
10 & 0 & 0 & 0 & 0 & 0 & 0 & 0 & 0 & 0 & 0 & 0 & 0 & 0 & 0 & 2 \\
\hline
11 & 0 & 0 & 0 & 0 & 0 & 0 & 0 & 0 & 0 & 0 & 0 & 0 & 0 & 0 & 0 
\end{tabular}
\caption{degree one map onto $3_1$}
\label{tabledegreeone31}
\end{table}

\begin{table}[h]
\begin{tabular}{c||c|c|c|c|c|c|c|c|c|c|c|c|c|c|c}
\diagbox{$g$}{$c$} & 17 & 18 & 19 & 20 & 21 & 22 & 23 & 24 & 25 & 26 & 27 & 28 & 29 & 30 & 31 \\
\hline
4 & 0 & 0 & 0 & 0 & 0 & 0 & 0 & 0 & 0 & 0 & 0 & 0 & 0 & 0 & 0 \\
\hline
5 & 0 & 0 & 0 & 0 & 0 & 0 & 0 & 0 & 0 & 0 & 0 & 0 & 0 & 0 & 0 \\
\hline
6 & 0 & 0 & 4 & 0 & 4 & 0 & 4 & 0 & 4 & 0 & 0 & 4 & 0 & 4 & 0 \\
\hline
7 & 0 & 0 & 4 & 4 & 12 & 8 & 20 & 12 & 28 & 16 & 36 & 20 & 44 & 24 & 52 \\
\hline
8 & 0 & 0 & 0 & 0 & 0 & 0 & 0 & 0 & 0 & 0 & 0 & 0 & 0 & 0 & 0 \\
\hline
9 & 0 & 0 & 0 & 0 & 0 & 0 & 0 & 0 & 0 & 0 & 0 & 0 & 0 & 0 & 4 \\
\hline
10 & 0 & 0 & 0 & 0 & 0 & 0 & 0 & 0 & 0 & 0 & 0 & 0 & 0 & 0 & 12 \\
\hline
11 & 0 & 0 & 0 & 0 & 0 & 0 & 0 & 0 & 0 & 0 & 0 & 0 & 0 & 0 & 12 \\
\hline
12 & 0 & 0 & 0 & 0 & 0 & 0 & 0 & 0 & 0 & 0 & 0 & 0 & 0 & 0 & 4 \\
\hline
13 & 0 & 0 & 0 & 0 & 0 & 0 & 0 & 0 & 0 & 0 & 0 & 0 & 0 & 0 & 0 
\end{tabular}
\caption{degree one map onto $6_2$}
\label{tabledegreeone62}
\end{table}

\section{Unknotting number one}

In this section, we consider a relationship between epimorphism of knot groups and their unknotting numbers. 
We have a problem that if there exists an epimorphism from $G(K)$ onto $G(K')$,  
then is the unknotting number of $K$ greater than or equal to that of $K'$? 
Some problems related to epimorphisms between knot groups are proposed in \cite{kitano-suzuki2}. 

We denote by $u(K)$ the unknotting number of a knot $K$. 
First, we recall the following theorem, which characterizes $2$-bridge knots with unknotting number one. 

\begin{theorem}[Kanenobu-Murakami \cite{KM}]\label{thm:km}
Let $K$ be a nontrivial $2$-bridge knot. 
Then the following three conditions are equivalent. 
\begin{enumerate}
\item $u(K)=1$. 
\item There exist an odd integer $p \, (> 1)$ and coprime, positive integeres $m$ and $n$ 
with $2mn = p \pm 1$ and $K$ is equivalent to $K(2 n^2/p)$. 
\item $K$ can be expressed as $K([a,a_1,a_2,\ldots,a_k,\pm 2, - a_k,\ldots,-a_2,-a_1])$. 
\end{enumerate}
\end{theorem}

%
%
%

By Theorem \ref{thm:km}, 
$2$-bridge knots of genus one with unknotting number one are 
$K([a, \pm 2])$ where $a$ is a positive even number. 
We can find $2$-bridge knots with unknotting number one 
whose knot groups admit epimorphism onto the knot group of $K([a, \pm 2])$. 

\begin{theorem}\label{thm:unknotting}
Let $K'$ be a $2$-bridge knot of $g(K')=1$ with $u(K')=1$, 
namely, $K'$ can be expressed as $K([a,\pm 2])$ where $a$ is a positive even number. 
There exist infinitely many $2$-bridge knots $K$ with $u(K)=1$ 
such that the knot groups $G(K)$ admit epimorphisms onto $G(K')$. 
Moreover, the minimum value of genera of such $2$-bridge knots is $3$. 
\end{theorem}

\begin{proof}
We can construct an infinite series of $2$-bridge knots $K_n$ with $u(K_n)=1$ 
such that $G(K_n)$ admits an epimorphism onto $G(K')$ as follows. 
For a natural number $n$, consider a $2$-bridge knot 
\[
K_n= K([\underbrace{a,\overbrace{\pm 2,a, \pm 2,a, \pm 2,\ldots,a}^{6n}, \pm 2}_{6n+2},
\underbrace{\overbrace{-a,\mp 2,-a,\mp2,-a,\mp2,\ldots,-a,\mp2}^{6n}}_{6n}]).
\]
By Theorem \ref{thm:km}, the unknotting number of $K_n$ is one. 
Since $a$ is an even number, 
this continued fraction expansion is of type $4n+1$ with respect to $(a,\pm 2)$  
and then $G(K_n)$ admits an epimorphism onto $G(K([a,\pm 2]))$. 

For the minimum genus, first we discuss $K' =K([a,2])$. 
The knot group of a $2$-bridge knot 
\[
K = K([\underbrace{a,2},2a,\underbrace{-2,-a},0,\underbrace{-a,-2}]) = K([a,\overbrace{2,2a},-2,\overbrace{-2a,-2}]) 
\]
admits an epimorphim onto $G(K([a,2]))$. 
The unknotting number of $K$ is $1$ and the genus is $3$. 
On the other hand, 
if the knot group of a $2$-bridge knot $K$ of genus $2$ admits an epimorphism onto $G(K[a,2])$ than 
\[
K = K([a,2,0,2,a,0,a,2]) = K([a,4,2a,2]) 
\]
(see Theorem \ref{thm:generatingfunction2} and \cite{suzukitran}).
The corresponding rational numbers to $K$ with even numerators are 
\[
\frac{16 a + 6}{16 a^2 + 10 a +1}, \quad \frac{8 a^2 + 3 a}{16 a^2 + 10 a +1}. 
\]
Here $(16 a + 6)(8 a^2 + 3 a) = (16 a^2 + 10 a +1)(8a + 1) - 1$. 
If $u(K)= 1$, these numerators can be expressed as $2 n^2$ 
by Theorem \ref{thm:km}, that is, 
$16a' + 3$ or $16 a'^2 + 3 a'$ is a square number, where $a = 2 a'$. 
Suppose that $b^2 = 16 a' +3$. 
We consider $b = 16b', 16 b' + 1, 16 b' + 2, \ldots, 16 b' + 15$, then $b^2 \equiv 0,1,4,9 \mod 16$.  
This contradicts $16 a' + 3 \equiv 3 \mod 16$. 
Therefore $16a' + 3$ is not a square number. 
Moreover, $(4a')^2 < 16 a'^2 + 3 a' < (4 a' + 1)^2$ implies that $16 a'^2 + 3 a'$ is also not a square number. 
Hence,  the numerators cannot be expressed as $2 n^2$ and $u(K) \neq 1$. 

Next, we discuss $K' =K([a,-2])$. 
The knot group of a $2$-bridge knot 
\begin{align*}
K & = K([\underbrace{a,-2},2a-2,\underbrace{2,-a},0,\underbrace{-a,2}]) = K([a,-2,2a-2,2,-2a,2]) \\
& = K([a-1,1,0,1,2a-3,1,1,2a-2,1,1]) = K([a-1,2,2a-3,1,1,2a-2,2]) \\
& = K([a-1,\overbrace{2,2a-2},-2,\overbrace{-(2a-2),-2}])
\end{align*}
admits an epimorphism onto $G(K([a,-2]))$. The genus of $K$ is $3$. 
On the other hand, if the knot group of a $2$-bridge knot $K$ of genus $2$ admits an epimorphism onto $G(K([a,-2]))$, 
then 
\[
K = K([a,-2,0,-2,a,0,a,-2]) = K([a,-4,2a,-2]). 
\]
The corresponding rational numbers with even numerators are 
\[
\frac{176 a - 18}{176 a^2 - 62 a - 1}, \quad \frac{8 a^2 - 3 a}{176 a^2 - 62 a - 1}. 
\]
Here $(176 a - 18)(8 a^2 - 3 a) = (176 a^2 - 62 a - 1)(8a - 1) - 1$. 
Similarly, we can show that neither $(176(2a') - 18)/2$ nor $(8 (2a')^2 - 3 (2a'))/2$ is a square number.  
Then we obtain $u(K) \neq 1$. 
\end{proof}
 
Note that the genus of $K_n$ in the proof of Theorem \ref{thm:unknotting} is $6n+1$ and that 
the crossing number is 
\[
\left\{
\begin{array}{ll}
(6n+1) a + 2 (6n+1) - 1 = (6n+1) a + 12 n +1 & \mbox{ if } K' = K([a,2]) \\
(6n+1) a + 2 (6n+1) - 12n = (6n+1) a + 2 & \mbox{ if } K' = K([a,-2]) 
\end{array}
\right. .
\]
 
\begin{remark}
The knot group of a $2$-bridge knot 
\[
K = K([\underbrace{a,2},-a,\underbrace{2,a},-2,\underbrace{-a,-2}]) = K([2,\overbrace{a,2,-a},-2,\overbrace{a,-2,-a}]) 
\]
admits an epimorphism onto $G(K([a,2]))$.  
The crossing number of $K$ is $4a + 8 - 3=4a +5$. 
This knot $K$ seems to attain the minimum value of the crossing numbers of $K$ with $u(K)=1$ 
which admits an epimorphism onto $G(K([a,2]))$. 

Similarly, The knot group of a $2$-bridge knot 
\[
K = K([\underbrace{a,-2},a-2,\underbrace{2,-a},2,\underbrace{-a,2}]) 
\]
admits an epimorphim onto $G(K([a,-2]))$, whose crossing number of $K$ is $4a + 6 - 6=4a$.  
The mirror image of $K$ can be expressed as 
\begin{align*}
K([2,-a,2,-a,2,a-2,-2,a]) &= K([1,1,a-2,2,a-2,1,1,a-3,2,a-1]) \\
& = K([1,1,a-2,2,a-2,1,1,a-3,2,a-2,1]) \\
& = K([1,\overbrace{1,a-2,2,a-2},2,\overbrace{-(a-2),-2,-(a-2),-1}]).
\end{align*} 
Then we get $u(K)=1$. 
This knot $K$ seems to attain the minimum value of the crossing numbers of $K$ with $u(K)=1$ 
which admits an epimorphism onto $G(K([a,-2]))$.
\end{remark}

\begin{example}\label{ex:unknotting31}
We can get Table \ref{tableunkotting31} and Table \ref{tableunknotting62} by computer program, 
which shows the number of $2$-bridge knots with unknotting number one which admit epimorphisms 
onto $G(3_1)$ and $G(6_2)$ respectively.  

\begin{table}[h]
\begin{tabular}{c||c|c|c|c|c|c|c|c|c|c|c|c|c|c|c}
\diagbox{$g$}{$c$} & 8 & 9 & 10 & 11 & 12 & 13 & 14 & 15 & 16 & 17 & 18 & 19 & 20 & 21 & 22\\
\hline
1 & 0 & 0 & 0 & 0 & 0 & 0 & 0 & 0 & 0 & 0 & 0 & 0 & 0 & 0 & 0 \\
\hline
2 & 0 & 0 & 0 & 0 & 0 & 0 & 0 & 0 & 0 & 0 & 0 & 0 & 0 & 0 & 0 \\
\hline
3 & 0 & 0 & 1 & 0 & 1 & 0 & 0 & 0 & 0 & 0 & 0 & 0 & 0 & 0 & 0 \\
\hline
4 & 0 & 0 & 1 & 0 & 1 & 0 & 0 & 0 & 0 & 0 & 0 & 0 & 0 & 0 & 0 \\
\hline
5 & 0 & 0 & 0 & 0 & 0 & 0 & 0 & 0 & 0 & 0 & 1 & 0 & 1 & 0 & 0 \\
\hline
6 & 0 & 0 & 0 & 0 & 0 & 0 & 0 & 0 & 1 & 0 & 3 & 0 & 2 & 0 & 0 \\
\hline
7 & 0 & 0 & 0 & 0 & 0 & 0 & 0 & 0 & 1 & 0 & 2 & 0 & 1 & 0 & 0 \\
\hline
8 & 0 & 0 & 0 & 0 & 0 & 0 & 0 & 0 & 0 & 0 & 0 & 0 & 0 & 0 & 0 \\
\hline
9 & 0 & 0 & 0 & 0 & 0 & 0 & 0 & 0 & 0 & 0 & 0 & 0 & 0 & 0 & 1 \\
\hline
10 & 0 & 0 & 0 & 0 & 0 & 0 & 0 & 0 & 0 & 0 & 0 & 0 & 0 & 0 & 1 \\
\hline
11 & 0 & 0 & 0 & 0 & 0 & 0 & 0 & 0 & 0 & 0 & 0 & 0 & 0 & 0 & 0 
\end{tabular}
\caption{unknotting number one onto $3_1$}
\label{tableunkotting31}
\end{table}

\begin{table}[h]
\begin{tabular}{c||c|c|c|c|c|c|c|c|c|c|c|c|c|c|c}
\diagbox{$g$}{$c$} & 17 & 18 & 19 & 20 & 21 & 22 & 23 & 24 & 25 & 26 & 27 & 28 & 29 & 30 & 31 \\
\hline
4 & 0 & 0 & 0 & 0 & 0 & 0 & 0 & 0 & 0 & 0 & 0 & 0 & 0 & 0 & 0 \\
\hline
5 & 0 & 0 & 0 & 0 & 0 & 0 & 0 & 0 & 0 & 0 & 0 & 0 & 0 & 0 & 0 \\
\hline
6 & 0 & 0 & 0 & 0 & 0 & 0 & 0 & 0 & 0 & 0 & 0 & 0 & 0 & 0 & 0 \\
\hline
7 & 0 & 0 & 1 & 0 & 1 & 0 & 0 & 0 & 0 & 0 & 0 & 0 & 0 & 0 & 0 \\
\hline
8 & 0 & 0 & 0 & 0 & 0 & 0 & 0 & 0 & 0 & 0 & 0 & 0 & 0 & 0 & 0 \\
\hline
9 & 0 & 0 & 0 & 0 & 0 & 0 & 0 & 0 & 0 & 0 & 0 & 0 & 0 & 0 & 0 \\
\hline
10 & 0 & 0 & 0 & 0 & 0 & 0 & 0 & 0 & 0 & 0 & 0 & 0 & 0 & 0 & 0 \\
\hline
11 & 0 & 0 & 0 & 0 & 0 & 0 & 0 & 0 & 0 & 0 & 0 & 0 & 0 & 0 & 0 \\
\hline
12 & 0 & 0 & 0 & 0 & 0 & 0 & 0 & 0 & 0 & 0 & 0 & 0 & 0 & 0 & 0 \\
\hline
13 & 0 & 0 & 0 & 0 & 0 & 0 & 0 & 0 & 0 & 0 & 0 & 0 & 0 & 0 & 0 
\end{tabular}
\caption{unknotting number one onto $6_2$}
\label{tableunknotting62}
\end{table}

The following $2$-bridge knots with unknotting number one admit epimorphisms 
onto $6_2 = K([2,-2,2,2])$: 
\begin{align*}
&K([\underbrace{2,-2,2,2},-2,\underbrace{2,2,-2,2},-2,\underbrace{-2,2,-2,-2}]) \\
&= K([2,\overbrace{2,-2,2,2,-2,2},-2,\overbrace{-2,2,-2,-2,2,-2}]), \\
&K([\underbrace{2,-2,2,2},2,\underbrace{-2,-2,2,-2},-2,\underbrace{-2,2,-2,-2}]) \\ 
&=K([2,\overbrace{2,-2,2,2,2,-2},2,\overbrace{2,-2,-2,-2,2,-2}]).
\end{align*}
The both genera are $7$ and the crossing numbers are $19$ and $21$ respectively. 
\end{example}

\begin{proposition}\label{prop:gg}
\begin{enumerate}
\item 
For any natural number $g \geq 3$, 
there exists a pair of $2$-bridge knots $K,K'$ of $g(K)=g$ with $u(K)=u(K')=1$ 
such that $G(K)$ admits an epimorphism onto $G(K')$. 
\item 
For any natural number $g'$, 
there exists a pair of $2$-bridge knots $K,K'$ of $g(K')=g'$ with $u(K)=u(K')=1$ 
such that $G(K)$ admits an epimorphism onto $G(K')$. 
\end{enumerate}
\end{proposition}

\begin{proof}
\begin{enumerate}
\item We define ${\bf a} = (a_1, a_2,\ldots,a_{2g})$ as follows. 
If $g$ is odd, then, 
\begin{align*}
& a_1 = 2, \quad 
a_{2k} = 2 \quad (1 \leq k \leq \frac{g-1}{2}),  \quad 
a_{2k+1} = 4 \quad (1 \leq k \leq \frac{g-1}{2}), \\
& a_{g+1} = (-1)^{(g-1)/2} \cdot 2, \quad  
a_{2g-k} = - a_{k+2} \quad (0 \leq k \leq g- 2) . 
\end{align*}
If $g$ is even, then, 
\begin{align*}
& a_1 = 2, \quad 
a_{2k} = 2 \quad (1 \leq k \leq \frac{g}{2}),  \quad 
a_{2k+1} = 4 \quad (1 \leq k \leq \frac{g-4}{2}), \\
& a_{g-1} = 2, \quad a_{g+1} = (-1)^{g/2} \cdot 2, \quad  
a_{2g-k} = - a_{k+2} \quad (0 \leq k \leq g- 2) . 
\end{align*}
Let $K = K([{\bf a}])$, then $u(K)=1$ and $g(K)=g$ by Theorem \ref{thm:km} and definition of ${\bf a}$. 
Remark that all $a_i$'s are non-zero even.  
We consider $K'= K([2,2]) = 4_1$, then $u(4_1) = 1$.  
The rational number $[{\bf a}]$ admits a continued fraction with respect to ${\bf a}$ 
of type $g$ if $g$ is odd, or of type $g-1$ if $g$ is even. 
Therefore there exists an epimorphism from $G(K)$ onto $G(K')$.

\item 
We define ${\bf a} = (a_1, a_2,\ldots,a_{2g'})$ by 
\begin{align*}
& a_{2k-1} = (-1)^{k+1} \cdot 2 \quad (1 \leq k \leq \frac{g'+1}{2}), \quad 
a_{2k} = (-1)^{k+1} \cdot 2 \quad (1 \leq k \leq \frac{g'}{2}), \\ 
& a_{2g' - (2k-1)} = (-1)^{k} \cdot 2 \quad (1 \leq k \leq \frac{g'}{2}), \quad 
a_{2g' - (2k-2)} = (-1)^{k+1} \cdot 2 \quad (1 \leq k \leq \frac{g'+1}{2}) .
\end{align*}
Let $K' = K([{\bf a}]) = K([{\bf a}^{-1}])$. 
Since all $a_i$'s are non-zero even, we see $ g(K')=g'$. 
We set 
\begin{align*}
K&= K([{\bf a}, 2,{\bf a}^{-1},-2,-{\bf a}]) \\
&= K([a_1,\overbrace{a_2,\ldots,a_{2g'},2,a_{2g'},\ldots,a_{g'+1}},a_g,
\overbrace{a_{g'-1},\ldots,a_1,-2,-a_1,\ldots,-a_{2g'}}]).
\end{align*}
Since $[{\bf a}, 2,{\bf a}^{-1},-2,-{\bf a}]$ is a continued fraction expansion of type $3$ with respect to ${\bf a}$, 
$G(K)$ admits an epimorphism onto $G(K')$. 
By definition, we have 
\[
a_{i} = - a_{2g'-i} ~ \, ~ (i=1,2,\ldots,g'-1), \quad a_{i} = a_{2g' - (i-2)} ~ \, ~ (i = 2,3,\ldots,2g'), \quad a_g' = \pm 2 . 
\]
Therefore, it turns out that $u(K)=  u(K') = 1$. 
\end{enumerate}
\end{proof}

\begin{remark}
It is assumed that the genus of $K$ is at least $3$ in Proposition \ref{prop:gg} (1). 
The previous paper \cite{suzukitran} shows that $2$-bridge knots of genus $1$ are always minimal.
Here a knot $K$ is called {\it minimal} if $G(K)$ admits an epimorphism 
only onto $G(K)$ and the knot group of the trivial knot. 
Moreover, $2$-bridge knots of genus $2$ are also minimal except for $K([2a,4b,4a,2b])$. 
Actually, $G(K([2a,4b,4a,2b]))$ admits an epimorphism onto $G([2a,2b])$. 
However, $K([2a,4b,4a,2b])$ has unknotting number at least $2$ for $|a|,|b| \leq 50$. 
\end{remark}

\begin{problem}
Let $K,K'$ be $2$-bridge knots with unknotting number one. 
If $G(K)$ admits an epimorphism onto $G(K')$, are the following inequalities 
\[
g(K) \geq 3 g(K'), \quad c(K) \geq 3 c(K') + 1 
\] 
satisfied?
\end{problem}

We can verify that 
any $2$-bridge knot groups with up to $50$ crossings and unknotting number one 
does not admit an epimorphism onto a $2$-bridge knot group with unknotting number more than one.  
Murasugi in \cite{murasugi3} showed that $2 u(K) \geq | \sigma (K) |$, 
where $\sigma (K)$ is the {\it signature} of a knot $K$. 
The existence of an epimorphism $G(K) \to G(K')$ does not imply $\sigma (K) \geq \sigma (K')$. 
For example, as mentioned in Example \ref{ex:3162}, there exists an epimorphism $G(10_{32}) \to G(3_1)$, however, 
$| \sigma (10_{32}) | = 0$ and $| \sigma(3_1) | = 2$. 
Moreover, for any $2$-bridge knot $K'$, we can construct an epimorphism $G(K) \to G(K')$ 
such that $\sigma (K)=0$. 
We state the following again as a concluding problem. 

\begin{problem}
Does the existence of an epimorphism from $G(K)$ onto $G(K')$ imply $u(K) \geq u(K')$, 
in particular, for $2$-bridge knots? 
\end{problem}

\section*{Acknowledgments}
The author wishes to express his thanks to Makoto Sakuma for helpful communications for Section \ref{sect:degone}. 
This work is partially supported by KAKENHI grant No.\ 20K03596 and 19H01785 
from the Japan Society for the Promotion of Science.

\end{document}